\documentclass[a4paper,10pt]{amsart}

\usepackage{amsmath,amssymb,amsthm}
\usepackage{enumitem}
\usepackage{mathtools}
\usepackage[all]{xy}
\usepackage{longtable}

\theoremstyle{plain}
\newtheorem{theorem}{Theorem}[section]
\newtheorem{proposition}[theorem]{Proposition}
\newtheorem{lemma}[theorem]{Lemma}

\theoremstyle{definition}
\newtheorem{definition}[theorem]{Definition}
\newtheorem*{acknowledgements}{Acknowledgements}

\theoremstyle{remark}
\newtheorem{conjecture}[theorem]{Conjecture}
\newtheorem{remark}[theorem]{Remark}
\newtheorem{notation}[theorem]{Notation}
\newtheorem{convention}[theorem]{Convention}
\newtheorem{claim}[theorem]{Claim}
\newtheorem{condition}[theorem]{Condition}

\numberwithin{equation}{theorem}

\DeclareMathOperator{\Pic}{Pic}
\DeclareMathOperator{\Proj}{Proj}
\DeclareMathOperator{\pr}{pr}
\DeclareMathOperator{\Gr}{Gr}
\DeclareMathOperator{\SG}{SG}
\DeclareMathOperator{\rank}{rank}
\DeclareMathOperator{\Aut}{Aut}
\DeclareMathOperator{\Bl}{Bl}
\DeclareMathOperator{\Sing}{Sing}

\DeclareMathOperator{\Chow}{Chow}
\DeclareMathOperator{\sub}{sub}
\DeclareMathOperator{\im}{Im}
\DeclareMathOperator{\OG}{OG}
\DeclareMathOperator{\Sym}{Sym}
\DeclareMathOperator{\gen}{gen}
\DeclareMathOperator{\spe}{sp}
\DeclareMathOperator{\Sp}{Sp}

\newcommand{\sE}{\mathcal{E}}
\newcommand{\sF}{\mathcal{F}}
\newcommand{\sG}{\mathcal{G}}

\newcommand{\sM}{\mathcal{M}}

\newcommand{\sS}{\mathcal{S}}
\newcommand{\sQ}{\mathcal{Q}}

\newcommand{\cO}{\mathcal{O}}

\newcommand{\fg}{\mathfrak{g}}
\newcommand{\fh}{\mathfrak{h}}

\newcommand{\bC}{\mathbf{C}}

\newcommand{\bP}{\mathbf{P}}
\newcommand{\bQ}{\mathbf{Q}}
\newcommand{\bR}{\mathbf{R}}
\newcommand{\bZ}{\mathbf{Z}}
\newcommand{\bO}{\mathbf{O}}

\DeclareMathOperator{\Pas}{\mathcal{P}}
\newcommand{\PF}{\Pas_{F_4}}
\newcommand{\PG}{\Pas_{A_1 \times G_2}}
\newcommand{\tPF}{\widetilde{\Pas}_{F_4}}
\newcommand{\tPG}{\widetilde{\Pas}_{A_1 \times G_2}}
\newcommand{\tX}{\widetilde X}
\newcommand{\blank}{{-}}

\newcommand{\dB}{
{
\SelectTips{}{12}
\objectmargin={0pt}
\objectheight={30pt}
\objectwidth={5pt}
\xygraph{!{<0cm,0cm>;<0.7cm,0cm>:<0cm,0.7cm>::}
\circ*=!D{\scriptstyle 1} -[r] \circ*=!D{\scriptstyle 2}  -|{\cdots}[r] \circ -[r]  \circ*=!D{\scriptstyle n-1}-@2{-}|(.6)@{>}[r] \circ*=!D{\scriptstyle n}
}}
}
\newcommand{\dC}{
{
\SelectTips{}{12}
\objectmargin={0pt}
\objectheight={30pt}
\objectwidth={5pt}
\xygraph{!{<0cm,0cm>;<0.7cm,0cm>:<0cm,0.7cm>::}
\circ*=!D{\scriptstyle 1} -[r] \circ*=!D{\scriptstyle 2}  -|{\cdots}[r] \circ -[r]  \circ*=!D{\scriptstyle n-1}-@2{-}|(.5)@{<}[r] \circ*=!D{\scriptstyle n}
}}
}
\newcommand{\dF}{
{
\SelectTips{}{12}
\objectmargin={0pt}
\objectheight={20pt}
\objectwidth={5pt}
\xygraph{!{<0cm,0cm>;<0.7cm,0cm>:<0cm,0.7cm>::}
\circ*=!D{\scriptstyle 1}-[r] \circ*=!D{\scriptstyle 2} -@2{-}|(.6)@{>}[r] \circ*=!D{\scriptstyle 3}-[r] \circ*=!D{\scriptstyle 4} 
}}
}
\newcommand{\dG}{
{
\SelectTips{}{12}
\objectmargin={0pt}
\objectheight={20pt}
\objectwidth={5pt}
\xygraph{!{<0cm,0cm>;<0.7cm,0cm>:<0cm,0.7cm>::}
\circ*=!D{\scriptstyle 1}-@3{-}|(.6)@{>}[r] \circ*=!D{\scriptstyle 2}
}}
}
\newcommand{\dAG}{
{
\SelectTips{}{12}
\objectmargin={0pt}
\objectheight={20pt}
\objectwidth={5pt}
\xygraph{!{<0cm,0cm>;<0.7cm,0cm>:<0cm,0.7cm>::}
\circ*=!D{\scriptstyle 0}-@0{-}[r] \circ*=!D{\scriptstyle 1}-@3{-}|(.6)@{>}[r] \circ*=!D{\scriptstyle 2}
}}
}

\AtBeginDocument{%
\def\MR#1{}
}

\setenumerate{
label=(\arabic*),
font=\normalfont
}

\title[Stability of tangent bundles]{Fano manifolds and stability of tangent bundles}
\author[A. KANEMITSU]{Akihiro KANEMITSU}
\date{\today}
\address{Department of Mathematics, Graduate school of Science, Kyoto University, Kyoto 606-8502, Japan}
\email{kanemitu@math.kyoto-u.ac.jp}
\thanks{The author is a JSPS Research Fellow and supported by the Grant-in-Aid for JSPS fellows (JSPS KAKENHI Grant Number 18J00681).}
\subjclass[2010]{14J45, 14J40, 14J60}

\keywords{Fano manifold, tangent bundle, stability, horospherical variety}

\begin{document}

\begin{abstract}
We determine the stability/instability of the tangent bundles of the Fano varieties in a certain class of two orbit varieties, which are classified by Pasquier in 2009.
As a consequence, we show that some of these varieties admit unstable tangent bundles, which disproves a conjecture on stability of tangent bundles of Fano manifolds.
\end{abstract}

\maketitle

\subsection{}
A \emph{Fano manifold} $X$ is, by definition, a smooth projective variety $X$ whose anti-canonical divisor $-K_X$ is ample.
From a differential geometric viewpoint, it is very important to detect which Fano manifolds admit K\"ahler-Einstein metrics and, after the celebrated works \cite{Tia97,Don02,Ber16,CDS15a,CDS15b,CDS15c,Tia15}, it has been accomplished that the existence of a K\"ahler-Einstein metric on a given Fano manifold $X$ is equivalent to a purely algebraic stability condition for $X$, called \emph{$K$-polystability}.

As is well-known, not every Fano manifold admits a K\"ahler-Einstein metric.
For example, Matsushima proved that the existence of a K\"ahler-Einstein metric on a Fano manifold $X$ implies the reductivity of the automorphism group of $X$ \cite{Mat57}.
The automorphism group of a Fano manifold $X$ is, however, not always reductive.
Also it is usually very difficult to determine whether or not a given Fano manifold $X$ admits a K\"ahler-Einstein metric, though the existence of such a metric is rephrased by the $K$-polystability of the Fano manifold $X$ in question.

Therefore it would be useful to study several variants of stability conditions on Fano manifolds.
As such a variant, stability of tangent bundles (in the sense of Mumford-Takemoto) has attracted several attention of researchers.
By the Kobayashi-Hitchin correspondence, the polystability of the tangent bundle is equivalent to the existence of a Hermitian-Einstein metric on the bundle \cite{Kob82,Lub83,Don85,UY86,Don87}, and hence the polystability of the tangent bundle is weaker than the existence of a K\"ahler-Einstein metric.
Also, thanks to its simplicity, the stability of the tangent bundle is rather easy to handle in the framework of algebraic geometry.
Moreover it is expected that, for a given Fano manifold $X$, the (in)stability of the tangent bundle reflects very well the geometry of $X$.
For example, a folklore conjecture\footnote{It is written in \cite[Remark 2.13]{Ara19} that this conjecture is due to Iskovskikh.} claims the following:
\begin{conjecture}[Stability of tangent bundles]\label{conj:stab}
 Let $X$ be a Fano manifold.
 Assume that the Picard number $\rho _X$ of  $X$ is one.
 Then the tangent bundle $\Theta_X$ is (semi)stable.
\end{conjecture}

Conjecture~\ref{conj:stab} has been confirmed in the following cases:
\begin{enumerate}
 \item $r_X =1$ \cite[Theorem~3]{Rei78}, where $r_X$ is the Fano index of $X$;
 \item Smooth complete intersections in the projective space $\bP^m$ \cite{Sub91} \cite[Corollary~1.5]{PW95};
 \item $\dim X \leq 5$ \cite{Hwa98,PW95};
 \item $\dim X =6$ (semistability of $\Theta_X$) \cite{Hwa98};
 \item $r_X \geq \dim X -2$ \cite{PW95};
 \item $r_X > \frac{\dim X+1}{2}$ and the fundamental divisor $H_X$ is very ample \cite{HM99b,Hwa01},
\end{enumerate}
where the \emph{Fano index} $r_X$ is defined as the largest integer that divides $-K_X$ in the Picard group $\Pic(X)$, and the \emph{fundamental divisor} $H_X$ is the divisor $-K_X/r_X$.
Note that, in \cite{PW95}, the case $r_X =n-2$ is studied under the condition that $\left|-K_X\right|$ contains enough smooth members, while the condition is lately proved to hold automatically by \cite{Amb99,Mel99} (cf. \cite{Muk89}).
See also \cite{Tia92,Hwa00,BS05,Bis10,Iye14,Liu18} for other related works on stability of tangent bundles.

It would be noteworthy that there is a variant of Conjecture~\ref{conj:stab}, which addresses (in)stability of tangent bundles for all Fano manifolds (not necessarily $\rho_X =1$); it is expected that the instability of $\Theta_X$ reflects the Mori-theoretic geometry of the variety $X$, and hence that it is ``realized'' by a Mori contraction $\pi \colon X \to Y$.
See, e.g., \cite{Ste96} and \cite[Conjecture~3.21]{Pet01} for this variant of Conjecture~\ref{conj:stab}.

\subsection{}
The purpose of this paper is to study Conjecture~\ref{conj:stab} for a certain class of Fano manifolds, which are classified by Pasquier.

Boris Pasquier, in his article \cite{Pas09}, classified Fano manifolds with the following condition:
\begin{condition}[Pasquier's condition]\label{cond:Pas}
$X$ is a (smooth) Fano manifold with $\rho _X =1$.
Under the natural action of the identity component $\Aut^0(X)$ of the automorphism group  of $X$, the variety $X$ decomposes into  two orbits $X^0 \bigsqcup Z$, where $X^0$ is the open orbit and $Z$ is the closed orbit.
Moreover, the blow-up $\Bl_Z X$ of $X$ along $Z$ is again an $\Aut^0(X)$-variety with two orbits $X^0$ and the exceptional divisor $E$.
\end{condition}
His result can be summarized as follows:
\begin{theorem}[{\cite[Theorem~0.2]{Pas09}}]\label{thm:Pas}
Let $X$ be a Fano manifold which satisfies Condition~\ref{cond:Pas}.
Then one of the following conditions holds:
\begin{enumerate}
 \item \label{thm:Pas_h}
 $X$ is a horospherical variety, $\Aut^0(X)$ is not reductive, and the isomorphic class of $X$ is uniquely determined by a triple $(D,\omega_Y, \omega_Z)$, where $(D,\omega_Y,\omega_Z)$ is one of the following triples:
  \begin{enumerate}[label=(\roman*)]
   \item \label{thm:Pas_h_bn} $(B_n,\omega_{n-1},\omega_n)$ ($n \geq 3$);
   \item \label{thm:Pas_h_b3} $(B_3,\omega_1,\omega_3)$;
   \item \label{thm:Pas_h_c} $(C_n,\omega_k,\omega_{k-1})$ ($n \geq 2$, $k \in  \{\,2,\dots, n \,\}$);
   \item \label{thm:Pas_h_f} $(F_4,\omega_2,\omega_3)$;
   \item \label{thm:Pas_h_g} $(G_2,\omega_1,\omega_2)$.
  \end{enumerate}
 \item \label{thm:Pas_f} $\Aut^0(X)$ is a semi-simple group of type $F_4$, and $X$ is isomorphic to an $F_4$-variety, which we will denote by $\PF$.
 \item \label{thm:Pas_g} $\Aut^0(X)$ is a semi-simple group of type $A_1 \times G_2$, and $X$ is isomorphic to an $A_1 \times G_2$-variety, which we will  denote by $\PG$.
\end{enumerate}
\end{theorem}

In particular, if $X$ is one of the manifolds as in Theorem~\ref{thm:Pas}~\ref{thm:Pas_h}, then $X$ does not admit K\"ahler-Einstein metrics since its automorphism group is not reductive.
On the other hand, with Conjecture~\ref{conj:stab} in his mind, one may expect the stability of the tangent bundles of all Fano manifolds in Theorem~\ref{thm:Pas}.
This expectation or Conjecture~\ref{conj:stab} is, however, no longer true in general; the purpose of this paper is to prove the following theorem:

\begin{theorem}[Stability/instability of tangent bundles]\label{thm:stab}
Let $X$ be a Fano manifold as in Theorem~\ref{thm:Pas}.
 \begin{enumerate}
  \item Then the tangent bundle of $X$ is not semistable in the following cases:
   \begin{itemize}
    \item Case~\ref{thm:Pas_h}~\ref{thm:Pas_h_bn} with $n \geq 4$;
    \item Case~\ref{thm:Pas_h}~\ref{thm:Pas_h_f}.
   \end{itemize}
  \item In the remaining cases, the tangent bundle of $X$ is stable.
 \end{enumerate}
\end{theorem}

\subsection{}
The present article is organized as follows:
Section~\ref{sect:prelim} presents preliminaries; we will review basic concepts of stability of vector bundles and foliations on varieties.
Then we will introduce a set $\Phi^G$ of foliations on $X$, and provide propositions which ensure that the set $\Phi^G$ is non-empty if $\Theta_X$ is unstable.
In Section~\ref{sect:Pas}, we will describe the geometry of the varieties as in Theorem~\ref{thm:Pas}, mainly based on \cite{Pas09} and \cite{GPPS19}.
In Section~\ref{sect:CF}, we will introduce \emph{canonical foliations} on the varieties as in Theorem~\ref{thm:Pas} and then provide a criterion of the stability of tangent bundles in terms of canonical foliations.
Then the proof of Theorem~\ref{thm:stab} will be performed.
In the last section, we will give several remarks.

\begin{convention}
Given a torsion free sheaf $\sE$ on a smooth projective variety $X$, we will consider its first Chern class $c_1(\sE)$ as an element of the Chow group $A^1(X)$.
If $A^1(X) \simeq \bZ$, then we will denote by $H_X$ the ample generator of $A^1(X)$ and we may identify the first Chern class with an integer.

For a vector bundle $\sE$ on a variety $X$, we will denote by $\bP(\sE)$ the Grothendieck projectivization of $\sE$,  i.e., $\bP(\sE) = \Proj(\bigoplus_{i\geq0} \Sym ^i\sE)$.
A morphism $f \colon X \to Y$ is called a \emph{projective bundle} if there is a vector bundle $\sE$ on $Y$ such that $X \simeq \bP(\sE)$ and the morphism $f$ is the natural projection.
We will denote by $\xi_\sE$ the relative tautological divisor of $\bP(\sE)$.
We will use the same convention for the projectivization of a vector space $V$.
Therefore $\bP(V)$ parametrizes the hyperplanes in $V$.
Set $\bP_{\sub}(V) \coloneqq \bP(V^{\vee})$, which parametrizes the $1$-dimensional subspaces in $V$.

For a smooth variety $X$ (resp.\ a smooth morphism $\pi \colon X \to Y$), we will denote by $\Theta _X$ (resp.\ $\Theta_\pi$) the tangent bundle of $X$  (resp. the relative tangent bundle of $\pi$). 
\end{convention}

\begin{acknowledgements}
The author wishes to express his gratitude to Professors Gianluca Occhetta and Luis Eduardo Sol\'a Conde for helpful discussions on the geometry of horospherical varieties.
He is also grateful to Professors Kento Fujita and Yuji Odaka for helpful comments and discussions on $K$-stability of Fano manifolds.
He also wishes to thank Professor Shigeru Mukai for various discussions, mainly for discussions on the geometry of a Mukai manifold $\PG$.
\end{acknowledgements}

\section{Preliminaries}\label{sect:prelim}

\subsection{Stability of vector bundles}
Here we briefly recall the concept of stability of vector bundles in the sense of Mumford-Takemoto.

Let $X$ be a smooth projective variety of dimension $n$, and fix an ample Cartier divisor $H$ on $X$.
Given a (nonzero) torsion-free sheaf  $\sE$ of rank $r$ on $X$, we define its \emph{first Chern class} $c_1(\sE)$ as the divisor class of the line bundle $(\bigwedge^r \sE)^{\vee \vee}$.
Then the \emph{slope} of $\sE$ with respect to the polarization $H$ is defined as the averaged degree of the first Chern class:
\[
\mu (\sE) \coloneqq \frac{c_1(\sE) \cdot H^{n-1}}{\rank \sE}.
\]
Recall that a subsheaf $\sF \subset \sE$ is said to be \emph{saturated} if the quotient $\sE/\sF$ is torsion free.

\begin{definition}[Stability of torsion free sheaves]
 Let $\sE$ be a torsion free sheaf on a smooth projective variety $X$.
Then the sheaf $\sE$ is called \emph{stable} (resp.\ \emph{semistable}) if, for any (nonzero) saturated subsheaf $\sF \subsetneq \sE$, the inequality $\mu (\sF) < \mu (\sE)$ (resp.\ $\mu (\sF) \leq \mu (\sE)$) holds.
\end{definition}

\begin{remark}[Polystability of sheaves]\label{rem:stab}
\hfill
\begin{enumerate}
\item Recall that a sheaf $\sE$ is called \emph{polystable} if $\sE \simeq \sE_1 \oplus \cdots \oplus \sE_k$ for some stable sheaves $\sE_i$ with $\mu(\sE_1) = \cdots = \mu(\sE_k)$.
By the definition, we have the following implications for a torsion free sheaf $\sE$:
\[
\text{stable} \Longrightarrow \text{polystable} \Longrightarrow \text{semistable}.
\]
Note that, for an indecomposable sheaf $\sE$, its stability is equivalent to the polystability.
\item Let $X$ be a Fano manifold with $\rho_X =1$.
In the rest of this paper, we will only consider the polarization on $X$ given by the fundamental divisor $H_X$.
Note that the tangent bundle $\Theta _X$ of $X$ is indecomposable by \cite[Proposition~3.1]{CP02}.
Thus, the stability of $\Theta_X$ is equivalent to its polystability, and hence it is also equivalent to the existence of a Hermitian-Einstein metric on $\Theta_X$ by the Kobayashi-Hitchin correspondence.
\end{enumerate}
\end{remark}

Let $\sE$ be a torsion free sheaf on $X$.
Assume that $\sE$ is not semistable.
Then, by the definition, $\sE$ admits a saturated subsheaf $\sF$ with $\mu (\sF) > \mu (\sE)$.
It is well known that the set
\[
\{\, \mu(\sF) \mid \text{$\sF$ is a subsheaf of $\sE$} \,\}
\]
is bounded from above, and there exists a unique maximal subsheaf $\sE_{\max}$, called the \emph{maximal destabilizing subsheaf} of $\sE$, which attains the maximum slope:
\[
\mu(\sE_{\max}) = \max \{\, \mu(\sF) \mid \text{$\sF$ is a subsheaf of $\sE$} \,\}.
\]

Note that, by the maximality of $\sE_{\max}$, the sheaf $\sE_{\max}$ is semistable and saturated in $\sE$.

\begin{remark}[Invariance under the action of the automorphism group]\label{rem:inv}
Let $X$ be a smooth projective variety and $\sE$ a torsion free sheaf on $X$.
An \emph{automorphism} of the pair $(X,\sE)$ is defined as a pair $(g,\varphi)$  of isomorphisms $g \colon X \to X$ and $\varphi \colon \sE \to g^* \sE$.
Then, by the uniqueness of the maximal destabilizing subsheaves, the subsheaf $\sE_{\max}$ is preserved by any automorphism of the pair $(X,\sE)$, i.e., $\sE_{\max} = g^* \sE_{\max}$ via the identification $\varphi$.
\end{remark}

\subsection{Foliations and algebraicity of leaves}
Now we restrict our attention to the case of tangent bundles and recall several concepts regarding foliations.

\begin{definition}[Foliation]
Let $X$ be a smooth projective variety.
A \emph{foliation} $\sF$ on $X$ is a saturated subsheaf $\sF \subset \Theta _X$ that is closed under the Lie bracket, i.e., $[\sF,\sF] \subset \sF$, where $[\blank,\blank]$ is the Lie bracket.

We will denote by $\Sing\sF$ the closed subset on which $\Theta_X/\sF$ is not locally free.
Thus, on $X\setminus \Sing \sF$, the sequence
\[
0 \to \sF \to \Theta_X \to \Theta_X/\sF \to 0
\]
is an exact sequence of \emph{vector bundles}. 

Let $\sF$ be a foliation or, more generally,  a subsheaf of $\Theta _X$.
Assume that an algebraic group $G$ acts on $X$.
Then $\sF$ is said to be \emph{$G$-invariant} if, for any $g \in G$, the image of the natural composition map
 \[
 \sF \to \Theta_X \to g^*\Theta_X
 \]
coincides with $g^*\sF$.
\end{definition}

Let $\sF$ be a foliation on $X$.
Then the Frobenius theorem on integrability says that, if $x \in X \setminus \Sing\sF$, then there exist an analytic open neighborhood $U$ of $x \in X$ and a closed analytic submanifold $L_U \subset U$ such that $\Theta_{L_U} = \sF|_{L_U}$ as subsheaves in $\Theta_X|_{L_U}$.
A (connected) complex manifold $Y$ or, more precisely, a pair $(Y,\iota)$ of a complex manifold $Y$ and a holomorphic map $\iota \colon Y \to X \setminus \Sing\sF$ is called an \emph{integral manifold} if $\iota$ is a one-to-one immersion and its differential $\mathrm{d}\iota$ identifies $\Theta_Y$ with $\iota^* \sF$.
Then the global version of the Frobenius theorem asserts that, for $x \in X\setminus\Sing \sF$,  there exists a unique maximal integral manifold $L_x$ that contains $x$, called \emph{leaf} through $x \in X$ (cf. \cite[Theorem 1.64]{War83}).
By an abuse of notation, we will denote by the same symbol $L_x$ the image of $L_x$ in $X$.
\begin{definition}[Leaves and their algebraicity]
\hfill
\begin{enumerate}
 \item A leaf $L$ is called \emph{algebraic} if $L$ is open in its Zariski closure in $X$.
 \item A foliation $\sF$ is said to be \emph{algebraically integrable} if a leaf $L$ through a general point $x \in X$ is algebraic.
\end{enumerate}
\end{definition}

As is observed by Miyaoka \cite{Miy87a,Miy87b}, the maximal destabilizing subsheaf of the tangent bundle $\Theta_X$ plays  an important role in the classification theory of algebraic varieties. 
An easy but fundamental observation is that the maximal destabilizing subsheaf  of $\Theta_X$ defines a foliation if it satisfies a weak positivity condition (see, e.g., \cite[Lemma~9.1.3.1]{SB92} for a proof):
\begin{proposition}[Maximal destabilizing subsheaves of tangent bundles]\label{prop:MDS}
Let $X$ be a smooth projective variety and $\Theta_{X,\max}$ the maximal destabilizing subsheaf of $\Theta _X$.
If $\mu (\Theta_{X,\max}) >0$, then $\Theta_{X,\max}$ defines a foliation.
\end{proposition}

Also, it was observed by Bogomolov-McQuillan or Bost \cite{BM01,Bos01} that the positivity of a foliation also ensures the algebraicity of leaves, which answers a question of Miyaoka \cite[Remark~8.9]{Miy87b} (cf.\ \cite{MP97} and \cite{SB92}):

\begin{theorem}[Algebraicity of leaves]\label{thm:AL}
Let $X$ be a smooth projective variety and $\sF$ a foliation.
Assume that a projective curve $C$ is contained in $X\setminus \Sing\sF$ and $\sF|_C$ is ample.
Then any leaf $L_x$ through $x \in C$ is algebraic.
\end{theorem}
See also \cite{KSCT07} for an account of the above theorem.

\subsection{Definition of $\Phi_G$}
Let $X$ be a Fano manifold (with $\rho_X =1$) and assume that $\Theta _X$ is not semistable.
Then, by Proposition~\ref{prop:MDS}, the maximal destabilizing subsheaf $\Theta_{X,\max}$ defines a foliation on $X$.
By the Mehta-Ramanathan restriction theorem, the restriction of $\Theta_{X,\max}$ to a general complete intersection curve $C = \bigcap_{i=1}^{n-1} D_i$, where $D_i \in \left|m_i H_X\right|$ for $m_i \gg0$, is still a semistable bundle on $C$.
Then, by Hartshorne's criterion \cite[Theorem~2.4]{Har71}, the bundle $\Theta_{X,\max}|_C$ is ample, and hence $\Theta_{X,\max}$ defines an algebraically integrable foliation.
Note that, by Remark~\ref{rem:inv}, this foliation is $\Aut(X)$-invariant.
Summarizing, we have a natural $\Aut(X)$-invariant algebraically integrable foliation $\Theta_{X,\max}$, if $\Theta _X$ is not semistable.
\begin{definition}[The set $\Phi^G$]
Let $X$ be a Fano manifold with $\rho_X=1$ and assume that an algebraic group $G$ acts on $X$.
Then we define:
\[
\Phi^G \coloneqq
\{\, \sF \subsetneq \Theta_X \mid \text{$\sF$ is a $G$-invariant algebraically integrable foliation}\,\}.
\]
Note that we assume $\sF \neq \Theta_X$ in the above definition.

Also, for a real number $a \in \bR$, we define $\Phi^{G}_{>a}$ (resp.\ $\Phi^{G}_{\geq a}$) as the subset of $\Phi^G$ consisting of the foliations with $\mu(\sF) >a$ (resp. $\mu(\sF) \geq a$).

We also define:
\[
 \phi^G \coloneqq \max \{\, \mu(\sF) \mid \sF \in \Phi_{G} \,\}.
\]
\end{definition}

Then, by the discussion above, we have the following:
\begin{proposition}[$\Phi^G_{> \mu(\Theta_X)}$ and semistability of $\Theta_X$]\label{prop:notss}
Let $X$ be a Fano manifold with $\rho_X =1$ and $G$ an algebraic group acting on $X$.
Assume that $\Theta _X$ is not semistable.
Then $\Phi^G_{> \mu(\Theta_X)}$ is non-empty.
\end{proposition}
The following proposition ensures that a similar set $\Phi^G_{\geq \mu(\Theta_X)}$ is also non-empty if $\Theta_X$ is not \emph{stable}:

\begin{proposition}[$\Phi^G_{\geq \mu(\Theta_X)}$ and stability of $\Theta_X$]\label{prop:nots}
Let $X$ and $G$ be as in Proposition \ref{prop:notss}.
Assume that $\Theta _X$ is not \emph{stable}.
Then $\Phi^G_{\geq \mu(\Theta_X)}$ is non-empty.
\end{proposition}

\begin{proof}
By Proposition~\ref{prop:notss}, we may assume that $\Theta _X$ is semistable but not stable.
By Remark \ref{rem:stab}, $\Theta_X$ is not polystable.
Set $\mu \coloneqq \mu(\Theta_X)$.

\begin{claim}
There exists a $G$-invariant semistable saturated subsheaf $\sE \subsetneq \Theta_X$ with $\mu(\sE) = \mu$.
\end{claim}
\begin{proof}[Proof of Claim]
Since $\Theta _X$ is semistable but not stable, we can find a stable saturated subsheaf $\sE_1 \subsetneq \Theta_X$ such that $\mu(\sE_1) = \mu$.
Note that $\sE_1$ is reflexive.
If $\sE_1$ is $G$-invariant, then we have nothing to prove.
Otherwise we can find an element $g_1 \in G$ such that $g_1^*\sE_1 \not \subset \sE_1$ and hence $ \sE_1 \cap g_1^*\sE_1 \neq g_1^*\sE_1$.

Consider the following exact sequence:
\[
0 \to \sE_1 \cap g_1^*\sE_1 \to \sE_1 \oplus g_1^*\sE_1 \to \sE_1 + g_1^*\sE_1\to 0.
\]
Since  $\sE_1 \oplus g_1^*\sE_1$ is reflexive and since $\sE_1 + g_1^*\sE_1$ is torsion free, it follows that $ \sE_1 \cap g_1^*\sE_1$ is reflexive.
Assume $ \sE_1 \cap g_1^*\sE_1 \neq 0 $. Then $\mu( \sE_1 \cap g_1^*\sE_1) < \mu$ by the stability of $g_1^*\sE_1$.
This implies that $\mu(\sE_1 + g_1^*\sE_1) > \mu$, which contradicts to the semistability of $\Theta_X$.
Thus we have $ \sE_1 \cap g_1^*\sE_1 = 0 $, and hence $\sE_1 + g_1^*\sE_1$ is a direct sum in $\Theta_X$.
Set $\sE_2 \coloneqq \sE_1 + g_1^*\sE_1$.

If $\sE_2$ is not $G$-invariant, then we can find $g_2 \in G$ such that $g_2^*\sE_1 \not \subset \sE_2$.
Then by considering the following sequence
\[
0 \to  \sE_2 \cap g_2^*\sE_1  \to \sE_2 \oplus g_2^*\sE_1 \to \sE_2 + g_2^*\sE_1\to 0,
\]
we see that $\sE_3 \coloneqq \sE_2 + g_2^*\sE_1$ is a direct sum in $\Theta_X$ again.
Eventually this procedure terminates, and produces a $G$-invariant polystable subsheaf $\sE \subset \Theta_X$ with $\mu(\sE) = \mu$.
Since $\Theta_X$ is not polystable, we have $\sE \neq \Theta_X$.

Finally we will prove that $\sE$ is saturated in $\Theta_X$.
Note that $\sE$ is reflexive since it is a direct sum of reflexive sheaves.
Let $\widehat \sE$ be the saturation of $\sE$ in $\Theta_X$.
Then the semistability of $\Theta _X$ and $\sE$ implies $\mu(\widehat \sE) = \mu$.
Thus the map $\sE \to \widehat \sE$ is a generic isomorphism between reflexive sheaves with same slopes, and hence it is an isomorphism
\end{proof}

By the above claim, the following set is non-empty:
\[
\Phi \coloneqq \{\, \sE \subsetneq \Theta_X \mid \text{$\sE$ is a $G$-invariant semistable saturated subsheaf with $\mu(\sE) = \mu$} \,\}.
\]
Let $\sF$ be an element of $\Phi$.
\begin{claim}
$\sF$ is an algebraically integrable foliation.
\end{claim}
\begin{proof}[Proof of Claim]
By Theorem~\ref{thm:AL}, it is enough to see that $\sF$ defines a foliation or, equivalently, that the O'Neil tensor $\wedge^2 \sF \to \Theta_X/\sF $, which is induced from the Lie bracket, is a zero map.
For this, we will follow the argument of \cite[Lemma~9.1.3.1]{SB92}.

Note that, since $\Theta_X/\sF$ is torsion free, the natural map $\Theta_X/\sF \to (\Theta_X/\sF)^{\vee\vee}$ is injective.
Hence it is enough to prove that there are no nontrivial morphisms $(\wedge^2\sF)^{\vee\vee } \to (\Theta_X/\sF)^{\vee\vee}$.
Since $\sF$ is semistable, the sheaf $(\wedge^2 \sF)^{\vee\vee}$ is a semistable sheaf whose slope is $2\mu$.
On the other hand, the semistability of $\Theta _X$ together with the fact $\mu (\Theta_X/\sF) =\mu$ implies that $\Theta_X/\sF$ is also semistable, and so is $(\Theta_X/\sF)^{\vee\vee}$.
Since $\mu >0$, we have $2\mu >\mu$.
Thus there are no nontrivial morphisms $(\wedge^2 \sF)^{\vee\vee} \to (\Theta_X/\sF)^{\vee\vee} $.
\end{proof}
This completes the proof.
\end{proof}

\section{Pasquier's classification and their geometry}\label{sect:Pas}
Here we recall descriptions of Fano varieties in Pasquier's classification (=Theorem~\ref{thm:Pas}), based on \cite{Pas09} or \cite{GPPS19}.
Almost all results in this section can be found in the literature, except for the explicit calculations on the geometry of $\PG$ and $\PF$ (= Propositions~\ref{prop:bl_PF} and \ref{prop:bl_PG}), for which we will include a proof for the convenience of readers.

\subsection{Preliminaries: associated triples and notations}
In what follows, we will employ basic terminologies regarding algebraic groups and Lie algebras.
Let $R$ be a root system with its associated Dynkin diagram $D(R)$ (with respect to a choice of a set $\Delta$ of simple roots).
Denote by $I$ the index set of $\Delta$.
Thus each $i \in I$ corresponds to an element in $\alpha_i \in \Delta$.
We will denote by $\omega_i$ the fundamental weight corresponding to $\alpha_i$.
\begin{definition}[Associated triples]\label{def:triples}
Let $X$ be a Fano manifold as in Theorem~\ref{thm:Pas}.
Then we define its \emph{associated triple} $(D,\omega_Y,\omega_Z)$, which consists of a Dynkin diagram $D$ and weights $\omega_Y$ and $\omega_Z$, as follows:
\begin{itemize}
 \item For varieties as in Theorem~\ref{thm:Pas}~\ref{thm:Pas_h}, we simply associate the triple indicated in the theorem.
 Hence $(D,\omega_Y,\omega_Z)$ is one of the following triples:
 \begin{enumerate}
  \item $(B_n,\omega_{n-1},\omega_n)$ ($n \geq 3$);
  \item $(B_3,\omega_1,\omega_3)$;
  \item $(C_n,\omega_k,\omega_{k-1})$ ($n \geq 2$, $k  \in \{\,2,\dots,n\,\}$);
  \item $(F_4,\omega_2,\omega_3)$;
  \item $(G_2,\omega_1,\omega_2)$.
 \end{enumerate}
 \item For $\PF$, set $(D,\omega_Y,\omega_Z) \coloneqq (F_4,\omega_1,\omega_3)$.
 \item For $\PG$, set $(D,\omega_Y,\omega_Z) \coloneqq (A_1\times G_2, \omega_1,\omega_0 + \omega_2)$.
\end{itemize}
\end{definition}

Here we use the following convention for the labeling of nodes of Dynkin diagrams:
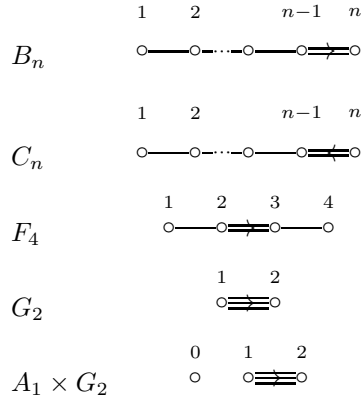
\begin{longtable}{lc}
 $B_n$ & $\dB$\\
 $C_n$ & $\dC$\\
 $F_4$ & $\dF$\\
 $G_2$ & $\dG$\\
 $A_1 \times G_2$ & $\dAG$\\
\end{longtable}

\begin{notation}
Let $(D,\omega_Y,\omega_Z)$ be an associated triple as in Definition~\ref{def:triples}.
In the following, $G$ denotes a (simply connected) semi-simple algebraic group whose Dynkin diagram is $D$.
More precisely, $G$ satisfies the following:
Let $\fg$ be the Lie algebra of $G$ and fix a Cartan subalgebra $\fh$ of $\fg$.
Then, by considering the Cartan decomposition of $\fg$ (with respect to $\fh$), we have the associated root system $R$ of $G$.
By choosing a set $\Delta$ of simple roots of $R$, we have the Dynkin diagram $D(R)$ of $R$.
Then we suppose the condition $D(R) = D$.

Furthermore, we use the following notations:
\begin{itemize}
 \item $V_Y$  (resp.\  $V_Z$) is the irreducible $G$-representation with highest weight $\omega_{Y}$ (resp.\ $\omega_Z$);
 \item $v_Y$ (resp.\ $v_Z$) is the corresponding highest weight vector;
 \item $[v_Y]$ (resp.\ $[v_Z]$) is the corresponding point in $\bP_{\sub}(V_{Y})$ (resp. $\bP_{\sub}(V_Z)$);
 \item $P_Y$ (resp.\ $P_Z$) is the stabilizer of $[v_Y]$ in $\bP_{\sub}(V_{Y})$ (resp.\ $[v_Z]$ in $\bP_{\sub}(V_{Z})$), which is the parabolic subgroup in $G$ corresponding to the weight $\omega_Y$ (resp.\ $\omega_Z$);
 \item $Y \coloneqq G/P_Y$ and $Z \coloneqq G/P_Z$ (later we will see that $Z$ is contained in $X$ as a closed orbit with respect to $\Aut^0(X)$, thus there will be no confusion with this notation, cf.\ Condition~\ref{cond:Pas});
 \item $P_{Y,Z}$ is the parabolic subgroup $P_Y \cap P_Z$;
 \item $p_Y \colon G/P_{Y,Z} \to G/P_Y =Y$ and $p_Z \colon G/P_{Y,Z} \to G/P_Z =Z$ are the natural projections.
\end{itemize}
Note that $P_Y$ is a \emph{maximal} parabolic subgroup, and this implies that $\Pic(Y) \simeq \bZ$.
We will denote by $\cO_Y(1)$ the ample generator of $\Pic(Y)$.
If $P_Z$ is also a \emph{maximal} parabolic subgroup, then we will use a similar notation.
\end{notation}

If $G$ is a semisimple algebraic group and $P$ is a maximal parabolic subgroup, then the Fano index of $G/P$ can be read off from the combinatoric data of the Dynkin diagram (see, e.g., \cite{Sno93}).

\subsection{Geometry of smooth horospherical varieties with Picard rank one}

Let $X$ be a Fano manifold as in Theorem~\ref{thm:Pas}~\ref{thm:Pas_h}.
We now turn to explain how we can recover the variety $X$ from the associated triple $(D,\omega_Y,\omega_Z)$. For more details we refer the reader to \cite{Pas09} and \cite{GPPS19}.

Note that, if $(D,\omega_Y,\omega_Z)$ is an  associated triple as in Theorem~\ref{thm:Pas}~\ref{thm:Pas_h}, then the projections $p_Y$ and $p_Z$ are projective space bundles.
Moreover, $p_Z^* \cO_{Z}(1)$ gives a relative tautological bundle of the projective bundle $p_Y$.
We will denote by $\sE_Y$ the bundle $(p_Y)_*p_Z^* \cO_{Y}(1)$.
Thus $G/P_{Y,Z}$ is isomorphic to the projectivization of $\sE_Y$ over $Y$.

\begin{proposition}[{\cite{Pas09}, \cite[Section~1.5]{GPPS19}}]\label{prop:bl_h}
Let $X$ be a Fano variety as in Theorem~\ref{thm:Pas}~\ref{thm:Pas_h} and $(D,\omega_Y,\omega_Z)$ the associated triple of $X$.
Then the following hold:
\begin{enumerate}
 \item $X$ is the $G$-orbit closure of the point $[v_Y \oplus v_Z] \in \bP_{\sub}(V_Y \oplus V_Z)$.
 \item The closed orbit of $X$ under the action of $\Aut^0(X)$ is isomorphic to $Z = G/P_Z$, which is naturally embedded into  $\bP_{\sub}(V_Y \oplus V_Z)$ as follows:
 \[
 G/P_Z = G\cdot[v_Z] \subset \bP_{\sub}(V_{Z}) \subset  \bP_{\sub}(V_Y \oplus V_Z).
 \]
 \item The blow up $\tX \coloneqq \Bl_Z X$ admits a contraction $\pi \colon \tX \to Y = G/P_Y$, which yields the  following diagram:
 \[
 \xymatrix{
 E \ar@{_{(}->}[d]  \ar[r]^-{\varphi|_E} & Z =G/P_Z \ar@{_{(}->}[d] \\
  \tX \ar[d]_-{\pi}  \ar[r]^-{\varphi} &  X\\
 Y  = G/P_Y,&
 }
 \]
 where $E$ is the exceptional divisor of the blow-up $\varphi$.
 \item $E \simeq G/P_{Y,Z}$.
 Moreover  $\pi|_E$ and $\varphi|_E$ are the natural projections $p_Y$ and $p_Z$.
 \item $\pi$ is a projective bundle given by $\bP(\sE_Y \oplus \cO_Y(1))$, and $E$ is the projective subbundle $\bP(\sE_Y) \subset \bP(\sE_Y \oplus \cO_Y(1))$.
 \item $\Pic (\tX) = \bZ \pi^*H_Y \oplus \bZ \varphi^*H_{X} $;
 \item $\xi_{\sE_Y \oplus \cO_Y(1)} = \varphi^*H_X$ and $E =\varphi^*H_X -\pi^* H_Y$.
\end{enumerate}
\end{proposition}

With the above proposition, one can compute several invariants of $X$.
We need the following data of numerical invariants for the proof of Theorem~\ref{thm:stab}:
\begin{proposition}[{see, e.g., \cite[Section~1.6]{GPPS19}}]\label{prop:bl_h_num}
Let $X$ be a Fano variety as in Theorem~\ref{thm:Pas}~\ref{thm:Pas_h} and $(D,\omega_Y,\omega_Z)$ the associated triple of $X$.
Then the numbers $\dim Y$, $c_1(Y)$,  $\dim Z$, $c_1(Z)$, $\dim X$ and $c_1(X)$ are as follows:
\renewcommand{\arraystretch}{1.8}
\begin{longtable}{l|cc}
 Associated triple& $\dim Y$ & $c_1(Y)$   \\
 \hline
 $(B_n,\omega_{n-1},\omega_{n})$ & $\dfrac{(n+4)(n-1)}{2}$& $n+1$ \\
 $(B_3,\omega_{1},\omega_{3})$ & $5$ & $5$ \\
 $(C_n,\omega_{k},\omega_{k-1})$ & $\dfrac{k(4n+1-3k)}{2}$& $2n+1-k$ \\
 $(F_4,\omega_{2},\omega_{3})$ &  $20$ & $5$ \\
 $(G_2,\omega_{1},\omega_{2})$ &  $5$ & $3$ \\
\end{longtable}
\begin{longtable}{l|cc}
 Associated triple&   $\dim Z$ & $c_1(Z)$  \\
 \hline
 $(B_n,\omega_{n-1},\omega_{n})$ & $\dfrac{n(n+1)}{2}$ & $2n$ \\
 $(B_3,\omega_{1},\omega_{3})$ & $6$ & $6$ \\
 $(C_n,\omega_{k},\omega_{k-1})$ & $\dfrac{(k-1)(4n+4-3k)}{2}$ & $2n+2-k$ \\
 $(F_4,\omega_{2},\omega_{3})$ & $20$ & $7$ \\
 $(G_2,\omega_{1},\omega_{2})$ & $5$ & $5$ \\
\end{longtable}
\begin{longtable}{l|cc}
 Associated triple&  $\dim X$ & $c_1(X)$ \\
 \hline
 $(B_n,\omega_{n-1},\omega_{n})$ & $\dfrac{n(n+3)}{2}$ & $n+2$ \\
 $(B_3,\omega_{1},\omega_{3})$ & $9$ & $7$ \\
 $(C_n,\omega_{k},\omega_{k-1})$ & $\dfrac{k(4n-3k+3)}{2}$ & $2n-k+2$ \\
 $(F_4,\omega_{2},\omega_{3})$ & $23$ & $6$ \\
 $(G_2,\omega_{1},\omega_{2})$ & $7$ & $4$ \\
\end{longtable}
\renewcommand{\arraystretch}{1}
\end{proposition}

\begin{remark}
Let $X$ be a Fano variety as in Theorem~\ref{thm:Pas}~\ref{thm:Pas_h} and $(D,\omega_Y,\omega_Z)$ be the associated triple.
Then the homogeneous variety $G/P_{Y,Z}$ admits two projective bundles $p_Y$ and $p_Z$, and $X$ is the \emph{smooth drum} associated to $G/P_{Y,Z}$ in the sense of \cite{ORSCW19}.
See \cite[Section~4]{ORSCW19} for a general treatment of drums and their relation to the torus actions.
\end{remark}

\subsection{Geometry of $\PF$}
Here we will describe the geometry of $\PF$ and its blow-up along the closed orbit.
We will denote by $F_4$ the exceptional group of type $F_4$ (with an abuse of notation).
Note that the closed orbit of $\PF$ is isomorphic to $Z = F_4/P_Z$ \cite[Proof of Proposition~2.13]{Pas09}.

Let $L(P_Y)$ be a Levi subgroup of $P_Y$.
Then $L(P_Y)$ is a reductive group whose commutator subgroup is a semisimple group of type $C_3$.
Recall that the simply connected semisimple algebraic group of type $C_3$ is isomorphic the symplectic group $\Sp(6)$ with respect to a symplectic vector space $(\bC^6,\omega \in \bigwedge^2 \bC^6)$.
Denote by $P(\omega_i)$ the parabolic subgroup corresponding to $\omega_i$.
Then, the homogeneous variety $C_3/P(\omega_1)$ is isomorphic to the symplectic Grassmann variety $\SG(6,1)$, which parametrizes the isotropic $1$-dimensional quotients of the symplectic vector space $(\bC^6,\omega)$.
Since every $1$-dimensional quotient of $\bC^6$ is isotropic, we have $\SG(6,1) \simeq \bP(\bC^6)$.

With this notation, the homogeneous variety $C_3/P(\omega_2)$ is naturally isomorphic to the symplectic Grassmann variety $\SG(6,2)$, which parametrizes the isotropic $2$-dimensional quotients of the symplectic vector space.
Then the variety $\SG(6,2)$ admits a natural embedding into the Grassmann variety $\Gr(6,2)$ and, under the Pl\"ucker embedding $\Gr(6,2) \to \bP(\wedge^2\bC^6)$, the subvariety $\SG(6,2)$ is the hyperplane section of $\Gr(6,2)$ corresponding to the symplectic form $\omega$.
Note that, under the natural action of $\Sp(6)$, the variety $\Gr(6,2)$ decomposes into two orbits $(\Gr(6,2)\setminus \SG(6,2)) \bigsqcup \SG(6,2)$.

\begin{proposition}[The blow-up of $\PF$]\label{prop:bl_PF}
The following hold:
\begin{enumerate}
 \item  The blow up $\tPF \coloneqq \Bl_Z \PF$ admits a contraction $\pi \colon \tPF \to Y = F_4/P_Y$, which yields the  following diagram:
 \[
 \xymatrix{
 E \ar@{_{(}->}[d]  \ar[r]^-{\varphi|_E} & Z = F_4/P_Z \ar@{_{(}->}[d] \\
  \tPF \ar[d]_-{\pi}  \ar[r]^-{\varphi} &  \PF \\
 Y  = F_4/P_Y,&
 }
 \]
 where $E$ is the exceptional divisor of the blow-up $\varphi$.
 \item $\pi$ is a smooth morphism whose fibers are isomorphic to $\Gr(6,2)$.
 \item $E$ is isomorphic to $F_4/P_{Y,Z}$.  Moreover  $\pi|_E$ and $\varphi|_E$ are the natural projections $p_Y$ and $p_Z$.
 \item $\dim \PF = 23$.
 \item $-K_{\PF} = 8 H_{\PF}$.
 \item $\Pic (\tPF) \simeq \bZ \pi^*H_Y \oplus \bZ \varphi^*H_{\PF} $.
 \item $E = - \pi ^* H_Y +  \varphi^* H_{\PF}$.
 \item $-K_{\pi} = -6 \pi ^* H_Y + 6 \varphi^* H_{\PF} = 6E$.
\end{enumerate}
\end{proposition}

\begin{proof}
Let $Q$ be the maximal parabolic subgroup corresponding to the fourth node of $F_4$
and denote by $f \colon F_4/(P_Y \cap Q) \to F_4/P_Y =Y$ and $g \colon F_4/(P_Y \cap Q) \to F_4/Q$ the natural projections.
Then any $f$-fiber is isomorphic to $C_3/P(\omega_1) \simeq \bP^5$, and $g^*\cO_{F_4/Q}(1)$ gives a relative tautological bundle of the $\bP^5$-bundle $f$.
Denote by $\sM$ the rank $6$ vector bundle $f_*g^*\cO_{F_4/Q}(1)$.
Then $F_4/(P_Y \cap Q) \simeq \bP(\sM)$ and the tautological divisor $\xi_\sM$ defines the contraction $g$.

Note that $g$ is a $\bQ^5$-bundle and $f^*H_Y$ restricts to the class of hyperplane section on each fiber $\bQ^5$, where $\bQ^5$ is the $5$-dimensional smooth hyperquadric.
Note also that $c_1(Y) =8$.
Thus, by adjunction and the canonical bundle formula for $\bP(\sM)$, we have
\[
-K_{\bQ^5} = -K_{\bP(\sM)}|_{\bQ^5} = (8-c_1(\sM)) H_{\bQ^5}.
\]
Thus we have $c_1(\sM) = 3$.

Now we will prove:
\begin{claim}
The blow up $\tPF$  is isomorphic to  the Grassmann variety $\Gr(\sM,2)$, which parametrizes two dimensional quotients of $\sM \otimes k(y)$ at each point $y \in Y$.
\end{claim}

\begin{proof}[Proof of Claim]
Since $\bP(\sM)$ is $F_4$-homogeneous, the variety $\Gr(\sM, 2)$ admits an action of $F_4$.
Note that, under the action of $F_4$, the variety $\Gr(\sM,2)$ contains $F_4/P_{Y,Z}$ as a closed orbit.
If we consider the subvariety $F_4/P_{Y,Z} \subset \Gr(\sM,2)$ as a divisor, we will denote it by $D$.

Let $\pi' $ be the natural projection $\Gr(\sM,2) \to Y$.
If we fix a point $y \in Y$ (corresponding to the unit of $F_4$), then the $\pi'$-fiber over $y$ is isomorphic to $\Gr(6,2)$, and this Grassmann variety admits an action of $P_Y$.
Then, under this action, the variety $\Gr(6,2)$ decomposes into two orbits $(\Gr(6,2)\setminus \SG(6,2)) \bigsqcup \SG(6,2)$.
This implies that $\Gr(\sM,2)$ is a two orbit variety with respect to the action of $F_4$, and the closed orbit is isomorphic to $F_4/P_{Y,Z}$.

Let $\sS$ (resp.\ $\sQ$) be the universal subbundle (resp.\ quotient bundle) on $\Gr(\sM,2)$.
Then the bundle $\det \sQ$ is a globally generated line bundle.
Note that $\det \sQ$ is not ample and restricts to $p_Z^* \cO_Z(1)$ on $F_4/P_{Y,Z}$.

Let $\eta$ be the divisor class corresponding to $\det \sQ$.
Then the line bundle $\det \sS$ corresponds to the divisor $-\eta + {\pi'} ^* c_1(\sM)$.
Since the relative tangent bundle $\Theta_{\pi'}$ is isomorphic to $\sQ \otimes \sS^{\vee}$, we have
\[
-K_{\pi'} = 6\eta - 2{\pi'}^*c_1(\sM) = 6\eta - 6{\pi'}^*H_Y.
\]
This implies that 
\[
-K_{\Gr(\sM,2)} = 6\eta + 2{\pi'}^*H_Y.
\]
Since $-K_{F_4/P_{Y,Z}} = 5p_Z^*H_Z + 3p_Y^*H_Y$, we see that the divisor $D$ on $\Gr(\sM,2)$ belongs to $|\eta - {\pi'}^*H_Y|$.
Therefore the restriction $-D|_D$ gives the relative tautological divisor for the $\bP^2$-bundle $p_Z$.
Thus the contraction $\varphi' \colon \Gr(\sM,2) \to X'$ defined by the divisor $\eta$ is the smooth blow-down along $D$, which is compatible with $p_Z$ \cite{Nak70,FN71}.
Therefore $X'$ satisfies Condition~\ref{cond:Pas} whose closed orbit is isomorphic to $Z$.
By Theorem~\ref{thm:Pas}, we have $X' \simeq \PF$.
\end{proof}
The rest of the assertion follows from what we have proved.
\end{proof}

\subsection{Geometry of $\PG$}
Similarly to the case of $\PF$, we will describe the geometry of $\PG$ and its blow-up along the closed orbit.
Here we also denote by $A_1\times G_2$ the simply connected semisimple algebraic group of type  $A_1\times G_2$ (with an abuse of notation).
Note that the closed orbit of $\PG$ is isomorphic to $Z =(A_1\times G_2) /P_Z \simeq  \bP^1 \times \bQ^5$ \cite[Proof of Proposition~2.13]{Pas09}.
In the following, we will denote by $K(G_2)$ the $5$-dimensional contact Fano manifold of type $G_2$, which is isomorphic to $Y= (A_1\times G_2) /P_Y$.

\begin{proposition}[The blow-up of $\PG$]\label{prop:bl_PG}
The following hold:
\begin{enumerate}
 \item  The blow up $\tPG \coloneqq \Bl_Z \PG$ admits a contraction $\pi \colon \tPG \to Y = (A_1\times G_2) /P_Y \simeq K(G_2)$, which yields the  following diagram:
 \[
 \xymatrix{
 E \ar@{_{(}->}[d]  \ar[r]^-{\varphi|_E} & Z \simeq\bP^1 \times \bQ^5 \ar@{_{(}->}[d] \\
 \tPG \ar[d]_-{\pi}  \ar[r]^-{\varphi} &  \PG \\
 Y = (A_1\times G_2) /P_Y \simeq K(G_2),&
 }
 \]
 where $E$ is the exceptional divisor of $\varphi$.
 \item $\pi$ is a $\bP^3$-bundle.
 \item $E$ is isomorphic to $ (A_1\times G_2) /P_{Y,Z}$ (the complete flag variety of type $A_1\times G_2$).  Moreover  $\pi|_E = p_Y$ and $\varphi|_E =p_Z$.
 \item $\dim \PG = 8$.
 \item $-K_{\PG} = 6 H_{\PG}$.
 \item $\Pic (\tPG) \simeq \bZ \pi^*H_Y \oplus \bZ \varphi^*H_{\PG} $.
 \item $E = - \pi ^* H_Y + 2 \varphi^* H_{\PG}$.
 \item $-K_{\pi} = -2 \pi ^* H_Y + 4 \varphi^* H_{\PG} =2E$.
\end{enumerate}
\end{proposition}

\begin{proof}
The proof proceeds similar to that of Proposition~\ref{prop:bl_PF}.
Let $Q_1$ (resp.\ $Q_2$) be the maximal parabolic subgroup corresponding to the first (resp.\ second) node of $G_2$.
Denote by $f \colon G_2/(Q_1 \cap Q_2) \to G_2/Q_1 \simeq Y$ and $g \colon G_2/(Q_1 \cap Q_2) \to G_2/Q_2 \simeq \bQ^5 $ the natural projections.
Then $f$ and $g$ are $\bP^1$-bundles.
Moreover $g^*\cO_{\bQ^5}(1)$ (resp.\  $f^*\cO_{Y}(1)$) gives a relative tautological bundle of $f$ (resp.\ $g$).
Denote by $\sM$ the rank $2$ vector bundle $f_*g^*\cO_{\bQ^5}(1)$.
Then $G_2/(Q_1\cap Q_2) \simeq \bP(\sM)$ and the tautological divisor $\xi_\sM$ defines the contraction $g$.
Note also that $c_1(Y) =3$ and $c_1(\sM) = 1$.

Now we prove:
\begin{claim}
$\tPG$  is isomorphic to $\bP(\sM \oplus \sM)$.
\end{claim}

\begin{proof}[Proof of Claim]

Since $\sM$ is a $G_2$-homogeneous vector bundle on $\sM$, the variety $\bP(\sM \oplus \sM)$ admits an action of $G_2$.
Moreover, by considering $\sM \oplus \sM$ as the tensor product $\cO_Y^{\oplus 2} \otimes \sM$, we have a natural action of $A_1 \times G_2$ on $\bP(\sM \oplus \sM)$.
Via the Segre embedding, we can equivariantly embed $\bP^1 \times  \bP(\sM) \simeq  (A_1\times G_2) /P_{Y,Z}$  into $\bP(\sM \oplus \sM)$.
Then one can check that, under the action of $A_1 \times G_2$, the variety $\bP(\sM \oplus \sM)$ decomposes into two orbits $\bP(\sM \oplus \sM) \setminus (\bP^1 \times  \bP(\sM)) \bigsqcup (\bP^1 \times  \bP(\sM))$.
When we consider the subvariety $\bP^1 \times  \bP(\sM)$ as a divisor, we will denote it by $D$.

Since $\sM$ is globally generated, the bundle $\sM \oplus \sM$ and hence the divisor $\xi_{\sM \oplus \sM}$ are globally generated.
Note that $\xi_{\sM\oplus\sM}$ restricts to $p_Z^* \cO_Z(1)$ on $\bP^1 \times  \bP(\sM) \simeq  (A_1\times G_2) /P_{Y,Z}$, where $\cO_Z(1)$ is the line bundle $\pr_1^*\cO_{\bP^1}(1) \otimes \pr_2^* \cO_{\bQ^5}(1)$ on $\bP^1 \times \bQ^5 =Z$.
We will denote by $H_Z$ the divisor class of the bundle $\cO_Z(1)$.

Let $\pi'$ be  the natural projection $\bP(\sM \oplus \sM) \to Y$.
By the canonical bundle formula for $\bP(\sM\oplus\sM)$, we have
\[
-K_{\pi'} = 4\xi_{\sM\oplus\sM}- {\pi'}^*c_1(\sM\oplus\sM)  = 4\xi_{\sM\oplus\sM}- 2{\pi'}^*H_Y.
\]
This implies that 
\[
-K_{\bP(\sM\oplus\sM)} = 4\xi_{\sM\oplus\sM} + {\pi'}^*H_Y.
\]
Since $-K_{(A_1\times G_2) /P_{Y,Z}} = 2 p_Y^*H_Y + 2p_Z^*H_Z$, we see that the divisor $D$ on $\bP(\sM\oplus\sM)$ belongs to $|2\xi_{\sM\oplus\sM} - {\pi'}^*H_Y|$.
Therefore the restriction $-D|_D$ gives a relative tautological divisor for the $\bP^1$-bundle $p_Z$.
Thus the contraction $\varphi' \colon  \bP(\sM\oplus\sM) \to  X'$ defined by the divisor $\xi_{\sM\oplus\sM}$ is the smooth blow-down which contracts the divisor $D$ compatibly with $p_Z$ \cite{Nak70,FN71}.
Therefore $X'$ satisfies Condition~\ref{cond:Pas} whose closed orbit is isomorphic to $Z$.
By Theorem~\ref{thm:Pas}, we have $X' \simeq \PG$.
\end{proof}
The rest of the assertion follows from the above claim.
\end{proof}

\begin{remark}[$\PG$ as a Mukai variety of genus seven]
\hfill
\begin{enumerate}
\item Originally the variety $\PG$ is defined as follows \cite[Definition~2.12]{Pas09}:
Let $\bO$ be the (complexified) octonions, and $\im \bO$ be the purely imaginary octonions.
We will denote by $(x \cdot y)$ the Cayley product on $\bO$.
Consider two elements $z_1$, $z_2 \in \im\bO$ such that $(z_1 \cdot z_1) = (z_1 \cdot z_2) =(z_2 \cdot z_2) =0$ and $[z_1] \neq [z_2]$ in $\bP_{\sub}(\im \bO)$.
Then the variety $\PG$ is the $(A_1 \times G_2)$-orbit closure of the point $[z_1 \oplus z_2]\in \bP_{\sub}(\im\bO \oplus \im\bO)$.

From this definition, one can show that
\[
\PG = \{\, [x_1\oplus x_2]  \in \bP_{\sub}(\im\bO \oplus \im\bO) \mid (x_1 \cdot x_1) = (x_1 \cdot x_2) =(x_2 \cdot x_2) =0 \,\}.
\]
It is well known (see, e.g., \cite[Example~2.15]{Tev05}) that (a component of)the orthogonal Grassmann variety $\OG(5,10)$ has a similar defining equation:
\[
\OG(5,10) = \{\, [x_1\oplus x_2]  \in \bP_{\sub}(\bO \oplus \bO) \mid (x_1\cdot \overline{x_1}) = (x_1\cdot \overline{x_2})= (x_2\cdot \overline{x_2}) \,\}.
\]
This implies that the variety $\PG$ is a codimension two linear section of $\OG(5,10)$, and hence the variety $\PG$ is a Mukai 8-fold of genus seven \cite{Muk89}.
Now the stability of $\Theta_{\PG}$ follows from \cite[Theorem~3]{PW95}.
Later we will provide a different proof of the stability of $\Theta_{\PG}$ based on our approach.

\item In \cite[Section~6]{Kuz18}, it is proved that there exist two isomorphic classes for Mukai $8$-folds with genus seven, and that one of these varieties $X_{\gen}$ degenerates to the other one $X_{\spe}$.
One can check that $\PG \simeq X_{\gen}$.
\end{enumerate}
\end{remark}

\section{Canonical foliations and stability of tangent bundles}\label{sect:CF}

\subsection{Canonical foliations}
Let $X$ be a Fano manifold as in Theorem~\ref{thm:Pas} and $Z$ its closed orbit (under the action of $\Aut^0(X)$).
Then, by Propositions~\ref{prop:bl_h}, \ref{prop:bl_PF} and \ref{prop:bl_PG}, we have the following diagram with a smooth morphism $\pi$:
\begin{equation}\label{eq:diagram}
\begin{gathered}
 \xymatrix{
 E \ar@{_{(}->}[d]  \ar[r]^-{\varphi|_E} & Z \ar@{_{(}->}[d] \\
 \tX = \Bl_Z X \ar[d]_-{\pi} \ar[r]^-{\varphi} & X \\
 Y. &
 }
\end{gathered}
\end{equation}

\begin{definition}[Canonical foliations]
The canonical foliation $\sF \subset \Theta _X$ is the foliation defined by the image of the map $\varphi_*\Theta_{\pi} \to \varphi_* \Theta_{\tX} \to \Theta_X$, i.e., the saturation of the image of $\varphi_*\Theta_{\pi}$ in $\Theta_X$.
\end{definition}

The following is a key lemma for Theorem~\ref{thm:stab}:

\begin{lemma}[A criterion of stability]\label{lem:stab}
Let $X$ be a Fano manifold as in Theorem~\ref{thm:Pas} and $\sF$ be the canonical foliation on $X$.
Then  $\Phi^{\Aut^0(X)} = \{\sF\}$.

Assume moreover that $\Theta_X$ is not stable (resp.\ not semistable).
Then $\Phi^{\Aut^0(X)}_{\geq \mu(\Theta_X)} = \{\sF\}$ (resp.\ $\Phi^{\Aut^0(X)}_{> \mu(\Theta_X)} = \{\sF\}$).
\end{lemma}

\begin{proof}
Take an element $\sF' \in \Phi^{\Aut^0(X)}$ and denote by $L_x$ the leaf of $\sF'$ through a general point $x \in X$.
Since $\sF'$ is $\Aut^0(X)$-invariant, the leaves are preserved by the action of $\Aut^0(X)$, i.e., $g(L_x)=L_{g(x)}$ for any $g \in \Aut^0(X)$.
In particular, $g(\overline{L_x})=\overline{L_{g(x)}}$ for any $g \in \Aut^0(X)$.

Consider the quotient map $f \colon X \dashrightarrow \Chow(X)$ with respect to the foliation $\sF'$, which sends a general point $x \in X$ to the point $\Bigl[ \overline{L_x} \Bigr]$.
Then the map $f$ is $\Aut^0(X)$-equivariant with respect to the natural actions on $X$ and $\Chow(X)$.

Thus, by the equivariant version of the Hironaka resolution of indeterminacy \cite{RY02} and by virtue of Condition~\ref{cond:Pas}, we can resolve the map $f$ after blowing up $X$ along $Z$.
Now the first assertion is clear.

Assume that $\Theta_X$ is not stable (resp.\ not semistable).
Then, by Proposition~\ref{prop:nots} (resp.\ Proposition~\ref{prop:notss}), the set $\Phi^{\Aut^0(X)}_{\geq \mu(\Theta_X)}$ (resp.\ $\Phi^{\Aut^0(X)}_{> \mu(\Theta_X)}$) is non-empty, and the last assertion follows.
\end{proof}

\subsection{Proof of Theorem~\ref{thm:stab}}
Recall that we have diagram \eqref{eq:diagram}.
Note also that, if $X$ is a variety as in Theorem~\ref{thm:Pas}~\ref{thm:Pas_h}, then there is a bundle $\sE_Y $  on $Y$ such that $\tX \simeq \bP(\sE_Y \oplus\cO_Y(1))$ (Proposition~\ref{prop:bl_h}).  
In the following, we will denote by $\sG \coloneqq \sE_Y \oplus \cO_Y(1)$.
\begin{proposition}\label{prop:CF_num}
Let $X$ be a Fano manifold as in Theorem~\ref{thm:Pas}~\ref{thm:Pas_h}, and $\sF$ its canonical foliation.
Then $\rank \sF = \rank \sE_Y$ and $c_1(\sF) = (\rank \sE_Y -c_1(\sE_Y))H_X$.
\end{proposition}

\begin{proof}
The first assertion is clear from the construction of the canonical foliation $\sF$.

We now turn to prove the second assertion.
By the canonical bundle formula for projective bundles, we have
\[
c_1(\Theta_{\pi}) = (\rank \sE_Y+1)\xi_\sG -(c_1(\sE_Y) +1)\pi^*H_Y.
\]
By Proposition~\ref{prop:bl_h}, 
\begin{align*}
 c_1(\Theta_{\pi}) &= (\rank \sE_Y+1)\xi_\sG -(c_1(\sE_Y)+1) \pi^*H_Y\\
 &= (\rank \sE_Y-c_1(\sE_Y))\varphi^*H_X +(c_1(\sE_Y)+1)E.
\end{align*}
Now the second assertion follows from the facts that the foliation $\sF$ is determined by $\Theta_\pi$ and that $\varphi$ is the blow-down of the divisor $E$.
\end{proof}

\begin{lemma}\label{lem:CF_num}
Let $X$ be a Fano manifold as in Theorem~\ref{thm:Pas}~\ref{thm:Pas_h}.
Then, for each associated triple, the numbers $\rank \sE_Y$ and $c_1(\sE_Y)$ are as follows:
\renewcommand{\arraystretch}{1}
\begin{longtable}{l|cc}
 Associated triple & $\rank \sE_Y$ & $c_1(\sE_Y)$ \\
 \hline
 $(B_n,\omega_{n-1},\omega_n)$ & $2$ & $1$ \\
 $(B_3,\omega_1,\omega_3)$ & $4$ & $2$ \\
 $(C_n,\omega_k,\omega_{k-1})$ & $k$ & $k-1$ \\
 $(F_4,\omega_2,\omega_3)$ & $3$ & $2$ \\
 $(G_2,\omega_1,\omega_2)$ & $2$ & $1$ \\
\end{longtable}
\renewcommand{\arraystretch}{1}
\end{lemma}

\begin{proof}
By Proposition~\ref{prop:bl_h}, the exceptional divisor $E$ is isomorphic to $G/P_{Y,Z} \simeq \bP(\sE_Y)$.
Thus $\rank \sE_Y = \dim X - \dim Y$.

By the canonical bundle formula for $G/P_{Y,Z} \simeq \bP(\sE_Y)$, we have
\begin{align*}
 -K_{G/P_{Y,Z}}& = (\rank\sE_Y) \xi_{\sE_Y} +(c_1(Y) - c_1(\sE_Y)) p_Y^*H_Y \\
 &= (\rank\sE_Y) p_Z^*H_Z +(c_1(Y) - c_1(\sE_Y))p_Y^*H_Y.
\end{align*}
Recall that $p_Z$ is a projective bundle.
Thus, by restricting $ -K_{G/P_{Y,Z}}$ to a $p_Z$-fiber $\bP^{\dim X- \dim Z - 1}$, we have
 \[
 \dim X- \dim Z =c_1(Y)-c_1(\sE_Y),
 \]
and the assertions follow from Proposition~\ref{prop:bl_h_num}.
\end{proof}

Now we can determine the rank and the first Chern class of the canonical foliations for all varieties as in Theorem~\ref{thm:Pas}:
\begin{proposition}\label{prop:CF}
Let $X$ be a Fano manifold as in Theorem~\ref{thm:Pas}, and $\sF$ its canonical foliation.
Then, for each associated triple, the numbers $\rank \sF$ and $c_1(\sF)$ are as follows:
\begin{longtable}{l|cc}
 Associated triple & $\rank \sF$ & $c_1(\sF)$ \\
 \hline
 $(B_n,\omega_{n-1},\omega_n)$ & $2$ & $1$ \\
 $(B_3,\omega_1,\omega_3)$ & $4$ & $2$ \\
 $(C_n,\omega_k,\omega_{k-1})$ & $k$ & $1$ \\
 $(F_4,\omega_2,\omega_3)$ & $3$ & $1$ \\
 $(G_2,\omega_1,\omega_2)$ & $2$ & $1$ \\
 \hline
 $(F_4,\omega_1,\omega_3)$ & $8$ & $0$ \\
 $(A_1\times G_2,\omega_1,\omega_0 + \omega_2)$ & $3$ & $0$ \\
\end{longtable}
\end{proposition}

\begin{proof}
For the case \ref{thm:Pas_h} of Theorem~\ref{thm:Pas}, the assertion follows from Proposition~\ref{prop:CF_num} and Lemma~\ref{lem:CF_num}.
For the remaining cases, the assertions are consequences of Proposition~\ref{prop:bl_PF} and Proposition~\ref{prop:bl_PG}.
\end{proof}

Now we can complete the proof of Theorem~\ref{thm:stab}.
\begin{proof}[Proof of Theorem~\ref{thm:stab}]
By Lemma~\ref{lem:stab}, it is enough to compare $\mu (\sF)$ with $\mu (\Theta_X)$ or, equivalently, $\dfrac{c_1(\sF)}{\rank\sF}$ with  $\dfrac{c_1(\Theta_X)}{\rank\Theta _X}$.
By Propositions~\ref{prop:bl_h_num}, \ref{prop:bl_PF}, \ref{prop:bl_PG} and \ref{prop:CF}, these numbers $\dfrac{c_1(\sF)}{\rank\sF}$ and $\dfrac{c_1(\Theta_X)}{\rank\Theta _X}$ are as follows, and we have the assertion:

\renewcommand{\arraystretch}{1.8}
\begin{longtable}{l|cc|c}
 Associated triple & $\dfrac{c_1(\sF)}{\rank\sF}$ & $\dfrac{c_1(\Theta_X)}{\rank\Theta _X}$ & $\dfrac{c_1(\sF)}{\rank\sF}> \dfrac{c_1(\Theta_X)}{\rank\Theta _X}$ ?\\
 \hline
 $(B_n,\omega_{n-1},\omega_n)$ & $1/2$ & $\dfrac{n+2}{n(n+3)/2}$ & {``$>$'' if and only if $n \geq 4$} \\
 $(B_3,\omega_1,\omega_3)$ & $1/2$ & $\dfrac{7}{9}$ & ``$<$'' \\
 $(C_n,\omega_k,\omega_{k-1})$ & $1/k$ & $\dfrac{2n-k+2}{k(4n-3k+3)/2}$ & ``$<$'' \\
 $(F_4,\omega_2,\omega_3)$ & $1/3$ & $\dfrac{6}{23}$ & ``$>$'' \\
 $(G_2,\omega_1,\omega_2)$ & $1/2$ & $\dfrac{4}{7}$ & ``$<$'' \\
 \hline
 $(F_4,\omega_1,\omega_3)$ & $0$ & $\dfrac{8}{23}$ & ``$<$'' \\
 $(A_1\times G_2,\omega_1,\omega_0 + \omega_2)$ & $0$ & $\dfrac{6}{8}$ & ``$<$'' \\
\end{longtable}
\renewcommand{\arraystretch}{1}
\end{proof}

\section{Several remarks}\label{sect:rem}

\begin{remark}[K\"ahler-Einstein metrics on $\PF$  and $\PG$]
 As mentioned in the introduction, every Fano manifold as in Theorem~\ref{thm:Pas}~\ref{thm:Pas_h} does not admit  K\"ahler-Einstein metrics.
By \cite[Corollary~5.7]{Del19}, they are not $K$-semistable.
On the other hand, for $X = \PF$ or $\PG$, this author does not know whether or not $X$ admits a K\"ahler-Einstein metric (and whether or not $X$ is $K$-semistable) at this moment.
\end{remark}

\begin{remark}[Fano foliations]
After a fundamental work on a characterizations of projective spaces and hyperquadrics \cite{ADK08}, Araujo and Druel started the study of Fano foliations in a series of papers \cite{AD14,AD16,AD17};
A foliation $\sF$ on a smooth projective variety $X$ is called \emph{Fano} if the first Chern class $c_1(\sF)$ is ample.
For such a foliation, its \emph{index} $r_\sF$ is defined as the largest integer which divides $c_1(\sF)$ in $\Pic(X)$.
As is similar to the case of the classification of Fano manifolds, the structure of a Fano foliation $\sF$ is rather simple if its index $r_\sF$ is relatively large with respect to $\rank \sF$.
For example, Fano foliations with $r_\sF \geq \rank \sF$ are completely classified \cite[Th\'eor\`eme~3.8]{DC05}, \cite{ADK08}.
The next largest index cases, i.e.,  the cases with $r_\sF = \rank \sF-1 $ and $r_\sF = \rank \sF-2 $ are called del Pezzo and Mukai foliations respectively, and these foliations are intensively studied in the above quoted papers \cite{AD14,AD16,AD17}.
See also \cite{Fig19} for a study of del Pezzo foliations, and \cite{Ara19} for an account of this topic on Fano foliations.

From this point of view, the canonical foliations on horospherical manifolds give several new examples of Fano, del Pezzo or Mukai foliations.
It would be noteworthy that, contrary to the horospherical cases, the canonical foliations on $\PF$ and $\PG$ have trivial first Chern classes, and hence $\phi^{\Aut^0(X)} = 0$.
\end{remark}

\bibliographystyle{amsalpha}
\bibliography{}

\providecommand{\bysame}{\leavevmode\hbox to3em{\hrulefill}\thinspace}
\providecommand{\MR}{\relax\ifhmode\unskip\space\fi MR }
\providecommand{\MRhref}[2]{%
  \href{http://www.ams.org/mathscinet-getitem?mr=#1}{#2}
}
\providecommand{\href}[2]{#2}
\begin{thebibliography}{ORSCW19}

\bibitem[AD14]{AD14}
Carolina Araujo and St\'{e}phane Druel, \emph{On codimension 1 del {P}ezzo
  foliations on varieties with mild singularities}, Math. Ann. \textbf{360}
  (2014), no.~3-4, 769--798. \MR{3273645}

\bibitem[AD16]{AD16}
\bysame, \emph{On {F}ano foliations 2}, Foliation theory in algebraic geometry,
  Simons Symp., Springer, Cham, 2016, pp.~1--20. \MR{3644241}

\bibitem[AD17]{AD17}
\bysame, \emph{Codimension 1 {M}ukai foliations on complex projective
  manifolds}, J. Reine Angew. Math. \textbf{727} (2017), 191--246. \MR{3652251}

\bibitem[ADK08]{ADK08}
Carolina Araujo, St\'{e}phane Druel, and S\'{a}ndor~J. Kov\'{a}cs,
  \emph{Cohomological characterizations of projective spaces and
  hyperquadrics}, Invent. Math. \textbf{174} (2008), no.~2, 233--253.
  \MR{2439607}

\bibitem[Amb99]{Amb99}
F.~Ambro, \emph{Ladders on {F}ano varieties}, J. Math. Sci. (New York)
  \textbf{94} (1999), no.~1, 1126--1135, Algebraic geometry, 9. \MR{1703912}

\bibitem[Ara19]{Ara19}
Carolina Araujo, \emph{Positivity and algebraic integrability of holomorphic
  foliations}, Proceedings of the International Congress of Mathematicians (ICM
  2018) (Boyan Sirakov, Paulo~Ney de~Souza, and Marcelo Viana, eds.),
  Proceedings of the International Congress of Mathematicians (ICM 2018),
  vol.~2, World Scientific, 2019, pp.~547--563.

\bibitem[Ber16]{Ber16}
Robert~J. Berman, \emph{K-polystability of {${\mathbb Q}$}-{F}ano varieties
  admitting {K}\"{a}hler-{E}instein metrics}, Invent. Math. \textbf{203}
  (2016), no.~3, 973--1025. \MR{3461370}

\bibitem[Bis10]{Bis10}
Indranil Biswas, \emph{Tangent bundle of a complete intersection}, Trans. Amer.
  Math. Soc. \textbf{362} (2010), no.~6, 3149--3160. \MR{2592950}

\bibitem[BM01]{BM01}
Fedor Bogomolov and Michael McQuillan, \emph{Rational curves on foliated
  varieties}, IHES preprint, 2001.

\bibitem[Bos01]{Bos01}
Jean-Beno\^{\i}t Bost, \emph{Algebraic leaves of algebraic foliations over
  number fields}, Publ. Math. Inst. Hautes \'{E}tudes Sci. (2001), no.~93,
  161--221. \MR{1863738}

\bibitem[BS05]{BS05}
Indranil Biswas and Georg Schumacher, \emph{On the stability of the tangent
  bundle of a hypersurface in a {F}ano variety}, J. Math. Kyoto Univ.
  \textbf{45} (2005), no.~4, 851--860. \MR{2226634}

\bibitem[CDS15a]{CDS15a}
Xiuxiong Chen, Simon Donaldson, and Song Sun, \emph{K\"{a}hler-{E}instein
  metrics on {F}ano manifolds. {I}: {A}pproximation of metrics with cone
  singularities}, J. Amer. Math. Soc. \textbf{28} (2015), no.~1, 183--197.
  \MR{3264766}

\bibitem[CDS15b]{CDS15b}
\bysame, \emph{K\"{a}hler-{E}instein metrics on {F}ano manifolds. {II}:
  {L}imits with cone angle less than {$2\pi$}}, J. Amer. Math. Soc. \textbf{28}
  (2015), no.~1, 199--234. \MR{3264767}

\bibitem[CDS15c]{CDS15c}
\bysame, \emph{K\"{a}hler-{E}instein metrics on {F}ano manifolds. {III}:
  {L}imits as cone angle approaches {$2\pi$} and completion of the main proof},
  J. Amer. Math. Soc. \textbf{28} (2015), no.~1, 235--278. \MR{3264768}

\bibitem[CP02]{CP02}
Fr\'{e}d\'{e}ric Campana and Thomas Peternell, \emph{Projective manifolds with
  splitting tangent bundle. {I}}, Math. Z. \textbf{241} (2002), no.~3,
  613--637. \MR{1938707}

\bibitem[DC05]{DC05}
Julie D\'{e}serti and Dominique Cerveau, \emph{Feuilletages et actions de
  groupes sur les espaces projectifs}, M\'{e}m. Soc. Math. Fr. (N.S.) (2005),
  no.~103, vi+124 pp. (2006). \MR{2200857}

\bibitem[Del19]{Del19}
Thibaut Delcroix, \emph{{$K$}-stability of {F}ano spherical varieties},
  arXiv:1608.01852v2, 2019.

\bibitem[Don85]{Don85}
S.~K. Donaldson, \emph{Anti self-dual {Y}ang-{M}ills connections over complex
  algebraic surfaces and stable vector bundles}, Proc. London Math. Soc. (3)
  \textbf{50} (1985), no.~1, 1--26. \MR{765366}

\bibitem[Don87]{Don87}
\bysame, \emph{Infinite determinants, stable bundles and curvature}, Duke Math.
  J. \textbf{54} (1987), no.~1, 231--247. \MR{885784}

\bibitem[Don02]{Don02}
\bysame, \emph{Scalar curvature and stability of toric varieties}, J.
  Differential Geom. \textbf{62} (2002), no.~2, 289--349. \MR{1988506}

\bibitem[Fig19]{Fig19}
Jo{\~a}o~Paulo Figueredo, \emph{{D}el {P}ezzo foliations with log canonical
  singularities}, arXiv:1909.03401v1, 2019.

\bibitem[FN72]{FN71}
Akira Fujiki and Shigeo Nakano, \emph{Supplement to ``{O}n the inverse of
  monoidal transformation''}, Publ. Res. Inst. Math. Sci. \textbf{7} (1971/72),
  637--644. \MR{0294712}

\bibitem[GPPS19]{GPPS19}
Richard Gonzales, Cl{\'e}lia Pech, Nicolas Perrin, and Alexander Samokhin,
  \emph{Geometry of horospherical varieties of {P}icard rank one},
  arXiv:1803.05063v2, 2019.

\bibitem[Har71]{Har71}
Robin Hartshorne, \emph{Ample vector bundles on curves}, Nagoya Math. J.
  \textbf{43} (1971), 73--89. \MR{292847}

\bibitem[HM99]{HM99b}
Jun-Muk Hwang and Ngaiming Mok, \emph{Varieties of minimal rational tangents on
  uniruled projective manifolds}, Several complex variables ({B}erkeley, {CA},
  1995--1996), Math. Sci. Res. Inst. Publ., vol.~37, Cambridge Univ. Press,
  Cambridge, 1999, pp.~351--389. \MR{1748609}

\bibitem[Hwa98]{Hwa98}
Jun-Muk Hwang, \emph{Stability of tangent bundles of low-dimensional {F}ano
  manifolds with {P}icard number {$1$}}, Math. Ann. \textbf{312} (1998), no.~4,
  599--606. \MR{1660263}

\bibitem[Hwa00]{Hwa00}
\bysame, \emph{Tangent vectors to {H}ecke curves on the moduli space of rank 2
  bundles over an algebraic curve}, Duke Math. J. \textbf{101} (2000), no.~1,
  179--187. \MR{1733732}

\bibitem[Hwa01]{Hwa01}
\bysame, \emph{Geometry of minimal rational curves on {F}ano manifolds}, School
  on {V}anishing {T}heorems and {E}ffective {R}esults in {A}lgebraic {G}eometry
  ({T}rieste, 2000), ICTP Lect. Notes, vol.~6, Abdus Salam Int. Cent. Theoret.
  Phys., Trieste, 2001, pp.~335--393. \MR{1919462}

\bibitem[Iye14]{Iye14}
Jaya N.~N. Iyer, \emph{Stability of tangent bundle on the moduli space of
  stable bundles on a curve}, arXiv:1111.0196v5, 2014.

\bibitem[Kob82]{Kob82}
Shoshichi Kobayashi, \emph{Curvature and stability of vector bundles}, Proc.
  Japan Acad. Ser. A Math. Sci. \textbf{58} (1982), no.~4, 158--162.
  \MR{664562}

\bibitem[KSCT07]{KSCT07}
Stefan Kebekus, Luis Sol\'{a}~Conde, and Matei Toma, \emph{Rationally connected
  foliations after {B}ogomolov and {M}c{Q}uillan}, J. Algebraic Geom.
  \textbf{16} (2007), no.~1, 65--81. \MR{2257320}

\bibitem[Kuz18]{Kuz18}
A.~G. Kuznetsov, \emph{On linear sections of the spinor tenfold. {I}}, Izv.
  Ross. Akad. Nauk Ser. Mat. \textbf{82} (2018), no.~4, 53--114. \MR{3833474}

\bibitem[Liu18]{Liu18}
Jie Liu, \emph{Stability of tangent bundles of complete intersections and
  effective restriction}, arXiv:1711.03413v3, 2018.

\bibitem[L{\"{u}}b83]{Lub83}
Martin L{\"{u}}bke, \emph{Stability of {E}instein-{H}ermitian vector bundles},
  Manuscripta Math. \textbf{42} (1983), no.~2-3, 245--257. \MR{701206}

\bibitem[Mat57]{Mat57}
Yoz\^{o} Matsushima, \emph{Sur la structure du groupe d'hom\'{e}omorphismes
  analytiques d'une certaine vari\'{e}t\'{e} k\"{a}hl\'{e}rienne}, Nagoya Math.
  J. \textbf{11} (1957), 145--150. \MR{94478}

\bibitem[Mel99]{Mel99}
Massimiliano Mella, \emph{Existence of good divisors on {M}ukai varieties}, J.
  Algebraic Geom. \textbf{8} (1999), no.~2, 197--206. \MR{1675146}

\bibitem[Miy87a]{Miy87a}
Yoichi Miyaoka, \emph{The {C}hern classes and {K}odaira dimension of a minimal
  variety}, Algebraic geometry, {S}endai, 1985, Adv. Stud. Pure Math., vol.~10,
  North-Holland, Amsterdam, 1987, pp.~449--476. \MR{946247 (89k:14022)}

\bibitem[Miy87b]{Miy87b}
\bysame, \emph{Deformations of a morphism along a foliation and applications},
  Algebraic geometry, {B}owdoin, 1985 ({B}runswick, {M}aine, 1985), Proc.
  Sympos. Pure Math., vol.~46, Amer. Math. Soc., Providence, RI, 1987,
  pp.~245--268. \MR{927960 (89e:14011)}

\bibitem[MP97]{MP97}
Yoichi Miyaoka and Thomas Peternell, \emph{Geometry of higher-dimensional
  algebraic varieties}, DMV Seminar, vol.~26, Birkh\"{a}user Verlag, Basel,
  1997. \MR{1468476}

\bibitem[Muk89]{Muk89}
Shigeru Mukai, \emph{Biregular classification of {F}ano {$3$}-folds and {F}ano
  manifolds of coindex {$3$}}, Proc. Nat. Acad. Sci. U.S.A. \textbf{86} (1989),
  no.~9, 3000--3002. \MR{995400}

\bibitem[Nak71]{Nak70}
Shigeo Nakano, \emph{On the inverse of monoidal transformation}, Publ. Res.
  Inst. Math. Sci. \textbf{6} (1970/71), 483--502. \MR{0294710}

\bibitem[ORSCW19]{ORSCW19}
Gianluca Occhetta, Eleonora~A. Romano, Luis~E. Sol{\'a}~Conde, and
  Jaros{\l}aw~A. Wi{\'s}niewski, \emph{Small bandwidth {${\mathbb
  C}^*$}-actions and birational geometry}, arXiv:1911.12129v1, 2019.

\bibitem[Pas09]{Pas09}
Boris Pasquier, \emph{On some smooth projective two-orbit varieties with
  {P}icard number 1}, Math. Ann. \textbf{344} (2009), no.~4, 963--987.
  \MR{2507635}

\bibitem[Pet01]{Pet01}
Thomas Peternell, \emph{Subsheaves in the tangent bundle: integrability,
  stability and positivity}, School on {V}anishing {T}heorems and {E}ffective
  {R}esults in {A}lgebraic {G}eometry ({T}rieste, 2000), ICTP Lect. Notes,
  vol.~6, Abdus Salam Int. Cent. Theoret. Phys., Trieste, 2001, pp.~285--334.
  \MR{1919461}

\bibitem[PW95]{PW95}
Thomas Peternell and Jaros\l aw~A. Wi\'{s}niewski, \emph{On stability of
  tangent bundles of {F}ano manifolds with {$b_2=1$}}, J. Algebraic Geom.
  \textbf{4} (1995), no.~2, 363--384. \MR{1311356}

\bibitem[Rei78]{Rei78}
Miles Reid, \emph{Bogomolov's theorem {$c\sb{1}\sp{2}\leq 4c\sb{2}$}},
  Proceedings of the {I}nternational {S}ymposium on {A}lgebraic {G}eometry
  ({K}yoto {U}niv., {K}yoto, 1977), Kinokuniya Book Store, Tokyo, 1978,
  pp.~623--642. \MR{578877}

\bibitem[RY02]{RY02}
Z.~Reichstein and B.~Youssin, \emph{Equivariant resolution of points of
  indeterminacy}, Proc. Amer. Math. Soc. \textbf{130} (2002), no.~8,
  2183--2187. \MR{1896397}

\bibitem[SB92]{SB92}
N.~I. Shepherd-Barron, \emph{{M}iyaoka's theorems on the generic seminegativity
  of {$T_X$} and on the {K}odaira dimension of minimal regular threefolds},
  Flips and abundance for algebraic threefolds (J{\'a}nos Koll{\'a}r, ed.),
  Ast\'{e}risque, no. 211, Soci\'{e}t\'{e} Math\'{e}matique de France, Paris,
  1992, pp.~103--114. \MR{1225842}

\bibitem[Sno93]{Sno93}
Dennis~M. Snow, \emph{The nef value and defect of homogeneous line bundles},
  Trans. Amer. Math. Soc. \textbf{340} (1993), no.~1, 227--241. \MR{1144015}

\bibitem[Ste96]{Ste96}
A.~Steffens, \emph{On the stability of the tangent bundle of {F}ano manifolds},
  Math. Ann. \textbf{304} (1996), no.~4, 635--643. \MR{1380447}

\bibitem[Sub91]{Sub91}
S.~Subramanian, \emph{Stability of the tangent bundle and existence of a
  {K}\"{a}hler-{E}instein metric}, Math. Ann. \textbf{291} (1991), no.~4,
  573--577. \MR{1135531}

\bibitem[Tev05]{Tev05}
E.~A. Tevelev, \emph{Projective duality and homogeneous spaces}, Encyclopaedia
  of Mathematical Sciences, vol. 133, Springer-Verlag, Berlin, 2005, Invariant
  Theory and Algebraic Transformation Groups, IV. \MR{2113135}

\bibitem[Tia92]{Tia92}
Gang Tian, \emph{On stability of the tangent bundles of {F}ano varieties},
  Internat. J. Math. \textbf{3} (1992), no.~3, 401--413. \MR{1163733}

\bibitem[Tia97]{Tia97}
\bysame, \emph{K\"{a}hler-{E}instein metrics with positive scalar curvature},
  Invent. Math. \textbf{130} (1997), no.~1, 1--37. \MR{1471884}

\bibitem[Tia15]{Tia15}
\bysame, \emph{K-stability and {K}\"{a}hler-{E}instein metrics}, Comm. Pure
  Appl. Math. \textbf{68} (2015), no.~7, 1085--1156. \MR{3352459}

\bibitem[UY86]{UY86}
K.~Uhlenbeck and S.-T. Yau, \emph{On the existence of
  {H}ermitian-{Y}ang-{M}ills connections in stable vector bundles}, Comm. Pure
  Appl. Math. \textbf{39} (1986), no.~S, suppl., S257--S293, Frontiers of the
  mathematical sciences: 1985 (New York, 1985). \MR{861491}

\bibitem[War83]{War83}
Frank~W. Warner, \emph{Foundations of differentiable manifolds and {L}ie
  groups}, Graduate Texts in Mathematics, vol.~94, Springer-Verlag, New
  York-Berlin, 1983, Corrected reprint of the 1971 edition. \MR{722297}

\end{thebibliography}
\end{document}